\documentclass[10pt,a4paper]{amsart}
\usepackage{amsmath,amsthm,amstext,amssymb,a4,hyperref}
\usepackage{graphics,epsfig}
\usepackage{psfrag}
\usepackage{graphicx}
\usepackage[T1]{fontenc}
\usepackage[]{algorithm2e}

\newcommand{\N}{\ensuremath{\mathbb{N}}}

\newcommand{\R}{\ensuremath{\mathbb{R}}}
\newcommand{\Z}{\ensuremath{\mathbb{Z}}}
\newcommand{\D}{\ensuremath{\mathrm{d}}}
\newcommand{\Li}{\ensuremath{\mathrm{Li}}}

\theoremstyle{plain}
\newtheorem{Thm}{Theorem}[]
\newtheorem{lemma}{Lemma}[]
\newtheorem{corollary}{Corollary}[]
\newtheorem{Prop}{Proposition}[]
\newtheorem{Def}{Definition}[]

\theoremstyle{remark}
\newtheorem{remark}{Remark}

\theoremstyle{definition}
\newtheorem{example}{Example}[]

\begin{document}

\author[]{A. Dirmeier}
\address{\normalfont \url{https://www.researchgate.net/profile/Alexander_Dirmeier}}
\email{\href{mailto:alexander.dirmeier@gmail.com}{alexander.dirmeier@gmail.com}}

\title[]{On Metrics Inducing the F\"urstenberg Topology on the Integers}

\begin{abstract}
\noindent We investigate various classes of metrics on the integers, which induce the F\"urstenberg topology and establish the connection between the metrics and the topology. We analyze the norm-like mappings underlying these metrics, with respect to their efficient computability for natural numbers and the analytic behavior of sequences under those mappings. Subsequently, we give some applications to number theory and establish some new propositions at the intersection of number theory and topology.
\end{abstract}

\maketitle


\section{Introduction}

This article is structured as follows. In section \ref{sec_one} we introduce the notions of the $q$-norms as norm-like mappings $\Z\to\R$ as well as the notion of the F\"urstenberg topology on $\Z$. We show that the $q$-norms imply metrics on $\Z$, which induce the F\"urstenberg topology. This has also been shown in \cite{Zulfeqarr2016}.

In section \ref{sec_two}, we give an algorithm for the efficient computation of the $q$-norms for natural numbers. We compute a value for the mean of the $q$-norms depending on $q>1$. Furthermore, we establish some facts about the distribution of the values of $q$-norms in $\R$.

Section \ref{sec_three} is about sequences of integers and conditions for their convergence in the $q$-norms as well as in the F\"urstenberg topology. We prove some connections between these two types of convergence.

In section \ref{sec_sum}, we establish the fact that summing or integrating the $q$-norms with respect to $q$ leads to new mappings $\Z\to\R$, which exhibit similar properties to those of the $q$-norms. In particular, those mappings also imply metrics on $\Z$, which induce the F\"urstenberg topology. But at the same time the summed or integrated $q$-norms have connections to arithmetical functions from number theory. Exploiting this fact, we are able to establish a new property of the Riemann zeta function and the number-of-divisors function.

A standard tool in analytic number theory are generating functions of arithmetical functions by means of Dirichlet series. In section \ref{sec_gen}, we compute the generating functions of several functions from the sections before and establish some of their properties.

\section{$q$-Norms and the F\"urstenberg Topology}\label{sec_one}

\begin{Def}\label{def_norm}
For $q\in(1,\infty)$ we call 
\[
 \|\cdot\|_q\colon \Z\to\R,
\]
defined by
\[
 \|n\|_q=\sum_{k\in\N,k\nmid n}\frac{1}{q^k}
\]
for $n\in\Z$ a \emph{$q$-norm} on the integers.
\end{Def}

\begin{remark}\label{rem_norm}
The functions $\|\cdot\|_q$ defined above fulfill 
\begin{itemize}
 \item[1.] $\|n\|_q\geq 0$ and $\|n\|_q=0\Leftrightarrow n=0$ as well as
 \item[2.] $\|n+m\|_q\leq\|n\|_q+\|m\|_q$
\end{itemize}
for all $n,m\in\Z$. Item 1 easily follows from $\frac{1}{q^k}>0$ for all $q\in(1,\infty)$ and all $k\in\N$ and the fact that for every $n\in\Z\setminus\{0\}$ there is at least one $k\in\N$ (e.g. $k=|n|+1$), which is not a divisor of $n$ and $\{k\in\N\,|\,k\nmid 0\}=\emptyset$. The triangle inequality in item 2 can be deduced from the following considerations. Certainly, the subset relation
\[
 \{k\in\N\,|\,k\mid n\wedge k\mid m\}\subset\{k\in\N\,|\,k\mid n+m\}
\]
holds for all $m,n\in\Z$. This yields the subset relation
\[
 \{k\in\N\,|\,k\nmid n\vee k\nmid m\}\supset\{k\in\N\,|\,k\nmid n+m\}
\]
and thus
\[
 \|m+n\|_q=\sum_{k\in\N,k\nmid m+n}\frac{1}{q^k}\leq \sum_{k\in\N,k\nmid n\vee k\nmid m}\frac{1}{q^k}\leq\sum_{k\in\N,k\nmid n}\frac{1}{q^k}+\sum_{k\in\N,k\nmid m}\frac{1}{q^k}=\|m\|_q+\|n\|_q.
\]
Note however that the $q$-norms are not \textit{norms} in the classical sense, as they are not positively homogeneous. But the properties of the $q$-norms in Remark \ref{rem_norm} are sufficient to yield a metric on the integers, which fulfills all classical properties of a distance. The case for $q=2$ was mentioned in \cite{Lovas2010} to be introduced in an online forum by J. Ferry (\cite{Ferry2009}). A different idea for the proof of the triangle inequality was given in \cite{Ferry2009} and it was suggested that the $q$-norm for $q=2$ has the properties analyzed below, but no explicit proof was given. However, an analysis equivalent to the one we will do below was also done in \cite{Zulfeqarr2016}.
\end{remark}

\begin{example}\label{ex_norm}
We use the following properties of the geometric series:
\[
 \sum_{k=0}^\infty\frac{1}{q^k}=\frac{1}{1-\frac{1}{q}}=\frac{q}{q-1}
\]
and
\[
 \sum_{k=i}^\infty\frac{1}{q^k}=\frac{q}{q-1}-\sum_{k=0}^{i-1}\frac{1}{q^k}=\frac{q}{q-1}-\frac{1-\frac{1}{q^i}}{1-\frac{1}{q}}
 =\frac{1}{q-1}\frac{1}{q^{i-1}}
\]
for all $i\in\N$. Then we have for example
\[
 \|1\|_q=\sum_{k=2}^\infty\frac{1}{q^k}=\frac{1}{q(q-1)}
\]
and
\[
 \|2\|_q=\sum_{k=3}^\infty\frac{1}{q^k}=\frac{1}{q^2(q-1)}.
\]
And for any prime number $p\in\N$
\[
 \|p\|_q=\sum_{k\in\N,k\nmid p}\frac{1}{q^k}=\sum_{k=2}^\infty\frac{1}{q^k}-\frac{1}{q^p}=\frac{1}{q(q-1)}-\frac{1}{q^p}
\]
holds, thus for any two prime numbers $p_1,p_2\in\N$, we have
\[
 p_1=p_2\Leftrightarrow\|p_1\|_q=\|p_2\|_q.
\]
Moreover
\[
 \|n!\|_q=\sum_{k\in\N,k\nmid n!}\frac{1}{q^k}\leq \sum_{k=n+1}^\infty\frac{1}{q^k}=\frac{1}{q-1}\frac{1}{q^n}
\]
for all $n\in\N$. Hence the $q$-norm of natural numbers with many divisors converges to zero for large numbers and the $q$-norm of large prime numbers converges to $\frac{1}{q(q-1)}$. More precisely, let $(p_n)_{n\in\N}$ be the increasing sequence of prime numbers in $\N$. Then we have
\[
 \|p_n\|_q\stackrel{n\to\infty}{\longrightarrow}\frac{1}{q(q-1)}
\]
and
\[
 \|n!\|_q\stackrel{n\to\infty}{\longrightarrow}0.
\]

\end{example}

\begin{Def}\label{def_metric}
Let $d_q\colon\Z\times\Z\to\R$ be the mapping given by
\[
 d_q(m,n)=\|m-n\|_q
\]
for all $m,n\in\Z$.
\end{Def}

\begin{Prop}\label{prop_metric}
The mapping $d_q$ given in Definition \ref{def_metric} is a metric on $\Z$. 
\end{Prop}

\begin{proof}
As $\|n\|_q\geq0$ for all $n\in\Z$, we also have $d_q(m,n)\geq 0$ for all $m,n\in\Z$. Furthermore
\[
 d_q(m,n)=0\Leftrightarrow\|m-n\|_q=0\Leftrightarrow m-n=0\Leftrightarrow m=n.
\]
As $m-n$ and $n-m$ have the same set of divisors, we also have
\[
 d_q(m,n)=\|m-n\|_q=\|n-m\|_q=d_q(n,m).
\]
The triangle inequality for $d_q$ follows from the triangle inequality of $\|\cdot\|_q$ proven in Remark \ref{rem_norm} due to
\[
 d_q(m,n)=\|m-n\|_q=\|m-l+l-n\|_q\leq\|m-l\|_q+\|l-n\|_q=d_q(m,l)+d_q(l,n) 
\]
for all $m,n,l\in\Z$.
\end{proof}

\begin{remark}
A proof of Prop.~\ref{prop_metric} is also contained in \cite[Thm.~3.2]{Zulfeqarr2016}. 
\end{remark}

\begin{example}\label{ex_metric}
The integers $1$ and $-1$ are the only numbers, which are only divisible by the natural number $1$. Thus
\[
 \max\{\|n\|_q\,|\,n\in\Z\}=\frac{1}{q(q-1)}=\|1\|_q=\|-1\|_q
\]
and also
\[
 \max\{d_q(m,n)\,|\,m,n\in\Z\}=\frac{1}{q(q-1)}=\|(k+1)-k\|_q=d_q(k+1,k)
\]
for any $k\in\Z$. Hence, the diameter of the metric space $(\Z,d_q)$ is $\frac{1}{q(q-1)}$ and it is realized by any two successive integers.
\end{example}

\begin{Def}\label{def_top}
The F\"urstenberg topology $\mathcal{T}$ on the integers $\Z$ is given by defining a set $A\subset\Z$ as open if for all $a\in A$ there is some $b\in\N$ such that $a+bn\in A$ for all $n\in\N$.
\end{Def}

\begin{remark}
A basis of the topology $\mathcal{T}$ is given by the open sets of arithmetic progressions
\[
 a+b\Z=\{a+bn\,|\, n\in\Z\}.
\]
Because of this, the topology $\mathcal{T}$ is also called topology of arithmetic progression in \cite{Lovas2010}. Therein also a metric based on the ``norm''
\[
 \|n\|=\frac{1}{\max\{k\in\N\,|\,1,\dots,k\mid n\}}
\]
is shown to yield the topology $\mathcal{T}$. Some properties of the topology $\mathcal{T}$ are also given in \cite[Ex. 58]{Steen1978}, where it is called evenly spaced integer topology. H. F\"urstenberg introduced the topological space $(\Z,\mathcal{T})$ in \cite{Furstenberg1955} and used it to give a topological proof for the fact that there are infinitely many prime numbers. In \cite{Broughan2003} the F\"urstenberg topology $\mathcal{T}$ was also analyzed under the name \emph{full topology}. Therein already was given a more general proof that a function similar to the ``norm'' above generates the F\"urstenberg topology (cf.~\cite[Thm.~3.1]{Broughan2003}).  
\end{remark}

\begin{lemma}\label{thelemma}
For all $a\in\Z$ and all $n\in\N$, we have
\[
 \|a\|_q<\frac{1}{q^n}\Rightarrow n\mid a.
\]
\end{lemma}

\begin{proof}
The contraposition 
\[
 n\nmid a\Rightarrow \|a\|_q\geq\frac{1}{q^n}
\]
follows easily from the definition of the norm:
\[
 \|a\|_q=\sum_{k\in\N,k\nmid a}\frac{1}{q^k}\geq\frac{1}{q^n}. 
\]
\end{proof}

\begin{Thm}\label{thetheorem}
The topology induced by the metrics $d_q$ is the F\"urstenberg topology $\mathcal{T}$.
\end{Thm}

\begin{proof}
Let $B_q(a,r)$ be the open $d_q$-ball of radius $r>0$ centered at some $a\in\Z$, i.e.,
\[
 B_q(a,r)=\{x\in\Z\,|\, d_q(a,x)<r\}.
\]
We have to show
\begin{itemize}
 \item[(i)] Any set $A\subset \Z$, which is open with respect to $d_q$, is also open with respect to $\mathcal{T}$ and
 \item[(ii)] Any set $A\subset \Z$, which is open with respect to $\mathcal{T}$, is also open with respect to $d_q$.
\end{itemize}

(i): Let $A\subset\Z$ be open with respect to $d_q$. Then for all $a\in A$ there is some $r>0$, such that $B_q(a,r)\subset A$. We show that in this case there exists some $b\in\N$, such that $a+b\Z\subset B_q(a,r)\subset A$, hence $A$ is also open with respect to $\mathcal{T}$.

Let $b=N!$ for some $N\in\N$, which fulfills $\frac{1}{q-1}\frac{1}{q^N}<r$. Then similarly to Example \ref{ex_norm} we have
\[
 d_q(a,a+bn)=\|bn\|_q=\sum_{k\in\N,k\nmid bn}\frac{1}{q^k}=\sum_{k\in\N,k\nmid n\cdot N!}\frac{1}{q^k}\leq\sum_{k=N+1}^\infty\frac{1}{q^k}=\frac{1}{q-1}\frac{1}{q^N}<r  
\]
for all $n\in\N$. Thus $a+b\Z\subset B_q(a,r)$.

(ii): Let $A\subset\Z$ be open with respect to $\mathcal{T}$. Then, by definition, for all $a\in A$ there exists some $b\in\N$, such that $a+b\Z\subset A$. We show that then there is some $r>0$, such that $B_q(a,r)\subset a+b\Z$, hence $A$ is also open with respect to $d_q$.

Let $r=\frac{1}{q^b}$ and let $x\in B_q(a,r)$, i.e., $d_q(a,x)=\|a-x\|_q<r=\frac{1}{q^b}$. Thus
\[
 \|a-x\|_q<\frac{1}{q^b},
\]
which yields by Lemma \ref{thelemma}
\[
 b\mid a-x.
\]
This implies $x=a+bn$ for some $n\in\Z$ and thus $x\in a+b\Z$. As this is true for any $x\in B_q(a,r)$, we have $B_q(a,r)\subset a+b\Z$.
\end{proof}

\begin{remark}
A proof of Thm.~\ref{thetheorem} is also contained in \cite[Thm.~4.1]{Zulfeqarr2016}. 
\end{remark}

Some more properties of the F\"urstenberg topology and the mappings called $q$-norms here, were analyzed in \cite{Lovas2015} and recently in \cite{Zulfeqarr2019}. Particularly, \cite{Lovas2015} also contains a different metric based on the factorial number system, which is then used to construct the completion of the topological space $(\Z,\mathcal{T})$.

\section{Efficient Computation of $q$-Norms and the Mean Value}\label{sec_two}

For several purposes it is convenient to have access to an abundance of numerical values for $q$-norms of natural numbers. The following Theorem \ref{thm_comp} gives an efficient algorithm for computing all $q$-norms for a given $q>1$ up to some $n\in\N$ together with an estimate of its computational complexity. 

\begin{Thm}\label{thm_comp}
The following algorithm computes the finite sequence
\[
 \left(\|1\|_q,\|2\|_q,\dots,\|n\|_q\right)
\]
for given $q>1$ and $n\in\N$ with a runtime in $\mathcal{O}(n\cdot\log(n))$.
\end{Thm}

\begin{algorithm}[H]
\KwData{$q>1$ and $n\in\N$}
\KwResult{the sequence of $q$-norms from $1$ to $n$}
$x := (q\cdot(q-1))^{-1}\cdot\textrm{\textbf{ones}}[n]$\;
\tcc*[f]{\textbf{ones}[n] returns an array of length n filled with ones}

\For{$i = 2$ \KwTo $n$}{ 
\For{$j = 1$ \KwTo $\lfloor\frac{n}{i}\rfloor$}{
$x[i\cdot j] := x[i\cdot j] - q^{-i}$\;
}
}
\KwRet $x$
\end{algorithm}

\begin{proof}
The number $1$ has $q$-norm $\frac{1}{q(q-1)}$, thus $x[1]$ has the correct value after initialization. All other natural numbers have $q$-norms smaller than $\frac{1}{q(q-1)}$, hence quantities---depending on the divisors of the respective numbers---have to be subtracted from the values $x[2]$ to $x[n]$, such that the correct $q$-norms of the numbers $2$ to $n$ are obtained. This is accomplished by the two nested for-loops. 

The outer for-loop runs over all numbers $i$, which are potential divisors of the numbers $2,\dots,n$. All numbers in $\{2,\dots,n\}$ divisible by $i$ have the form $i\cdot j$ with $1\leq j\leq\lfloor\frac{n}{i}\rfloor$, hence the inner for-loop runs over all such numbers $j$. Then the quantity $\frac{1}{q^i}$ is subtracted from the value at the position $i\cdot j$ in the array $x$, as for the number $i\cdot j$ the summand $\frac{1}{q^i}$ does \textit{not} appear in the definition of the $q$-norm.

As the numbers $2$ to $n$ can only have proper divisors between $2$ and $n$, all divisors for all numbers from $2$ up to $n$ are reached in the two for-loops and subtracted from the appropriate values. 

Note that for $i>\frac{n}{2}$ the inner for-loop is only executed once. This yields the following estimate for the runtime of this algorithm:

\[
 \underbrace{\frac{n}{2}+\frac{n}{3}+\dots+\frac{n}{\lfloor\frac{n}{2}\rfloor}}_{\textrm{approximate number of executions of}\atop\textrm{the inner for-loop for } 2\leq i\leq \lfloor\frac{n}{2}\rfloor} + \underbrace{\lfloor\frac{n}{2}\rfloor}_{\textrm{approximate number of executions of}\atop\textrm{the outer for-loop for } i>\frac{n}{2}} =
\]
\[
 n\sum_{k=2}^{\lfloor\frac{n}{2}\rfloor}\frac{1}{k} + \lfloor\frac{n}{2}\rfloor \in \mathcal{O}(c+n\cdot\log(n)+\frac{n}{2}) = \mathcal{O}(n\cdot\log(n))
\]
Where $c$ is some positive constant and we used the approximation 
\[
 \sum_{k=1}^{n}\frac{1}{k} = \log(n) + \gamma + \alpha(n)
\]
with the Euler-Mascheroni-constant $\gamma$ and some $\alpha(n)\in\mathcal{O}(\frac{1}{n})$.

\end{proof} 

\begin{Def}
Let
\[
 m_N(q) := \frac{\sum_{n=1}^N\|n\|_q}{N}
\]
be the mean value of all $q$-norms of the natural numbers up to $N\in\N$.
\end{Def}

Theorem \ref{thm_mean} below shows that the mean value $\lim_{N\to\infty}m_N(q)$ of the $q$-norm\linebreak $\|\cdot\|_q\colon\Z\to\R$ exists and depends on $q>1$. The following plot shows values of $\|n\|_q$ for $q=2$ and $1000\leq n\leq1100$ in blue as well as the mean value of $\|\cdot\|_2$, which in this case is $1-\ln(2)$ in red:
 
\includegraphics[width=\textwidth]{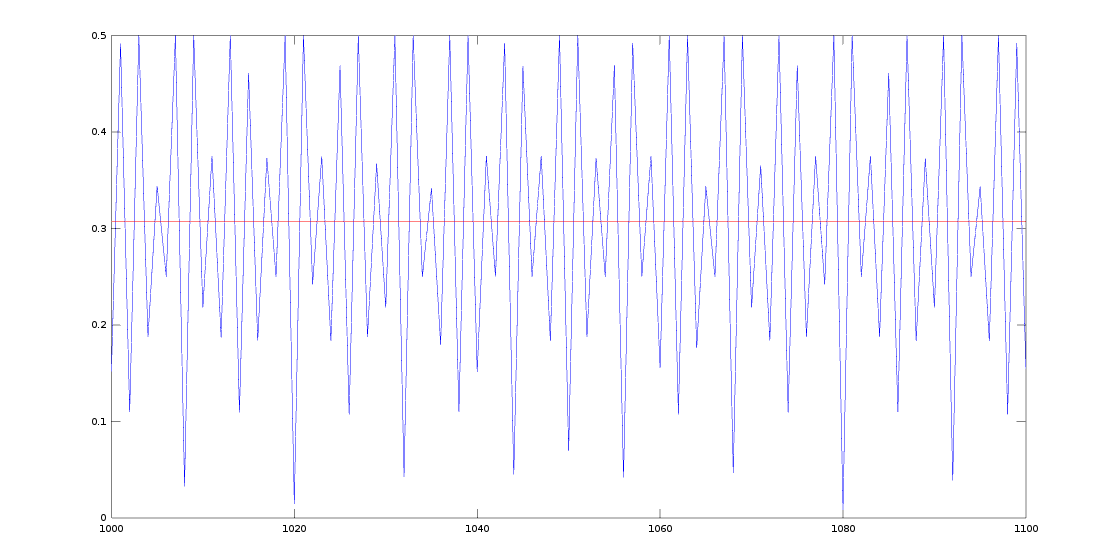} 

This plot suggests that there is some interval around the mean value $\lim_{N\to\infty}m_N(q)$, which is disjoint from the image of $\|\cdot\|_q$. In fact, plotting a histogram of the values $\|n\|_2$ for all $n\in\N$, $n\leq 10^6$ using $40$ classes, suggests that there are many such intervals disjoint from the image of $\|\cdot\|_q$. 

\includegraphics[width=\textwidth]{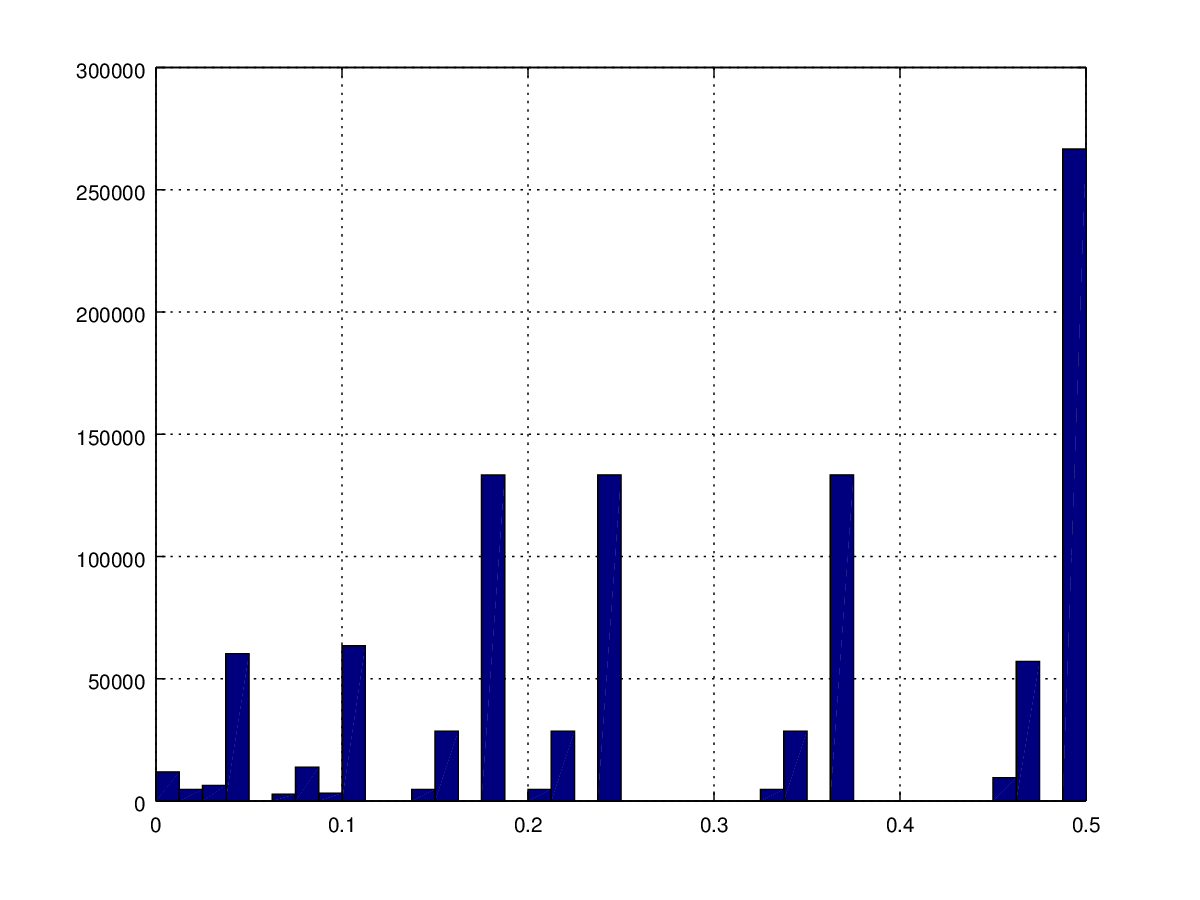}   

Or in other words: it seems the values of $\|\cdot\|_q$ do cluster in specific intervals in $[0,\frac{1}{q(q-1)}]$. The Theorem \ref{thm_mean} below, establishes a specific bound on $q>1$, for which there is a gap between the values $\|n\|_q$ for even $n\in\N$ and the values $\|n\|_q$ for odd $n\in\N$. This gap is an interval in $[0,\frac{1}{q(q-1)}]$ disjoint from the image of $\|\cdot\|_q$ for all $q>1$ obeying the bound. Furthermore, for another bound on $q>1$ the mean value $\lim_{N\to\infty}m_N(q)$ is located in this gap. The following plot shows the boundaries of this gap (for $q>\Phi$, c.f. Theorem \ref{thm_mean} below) in blue depending on $q>1$ and the mean value $\lim_{N\to\infty}m_N(q)$ for $q>1$ in red:

\includegraphics[width=\textwidth]{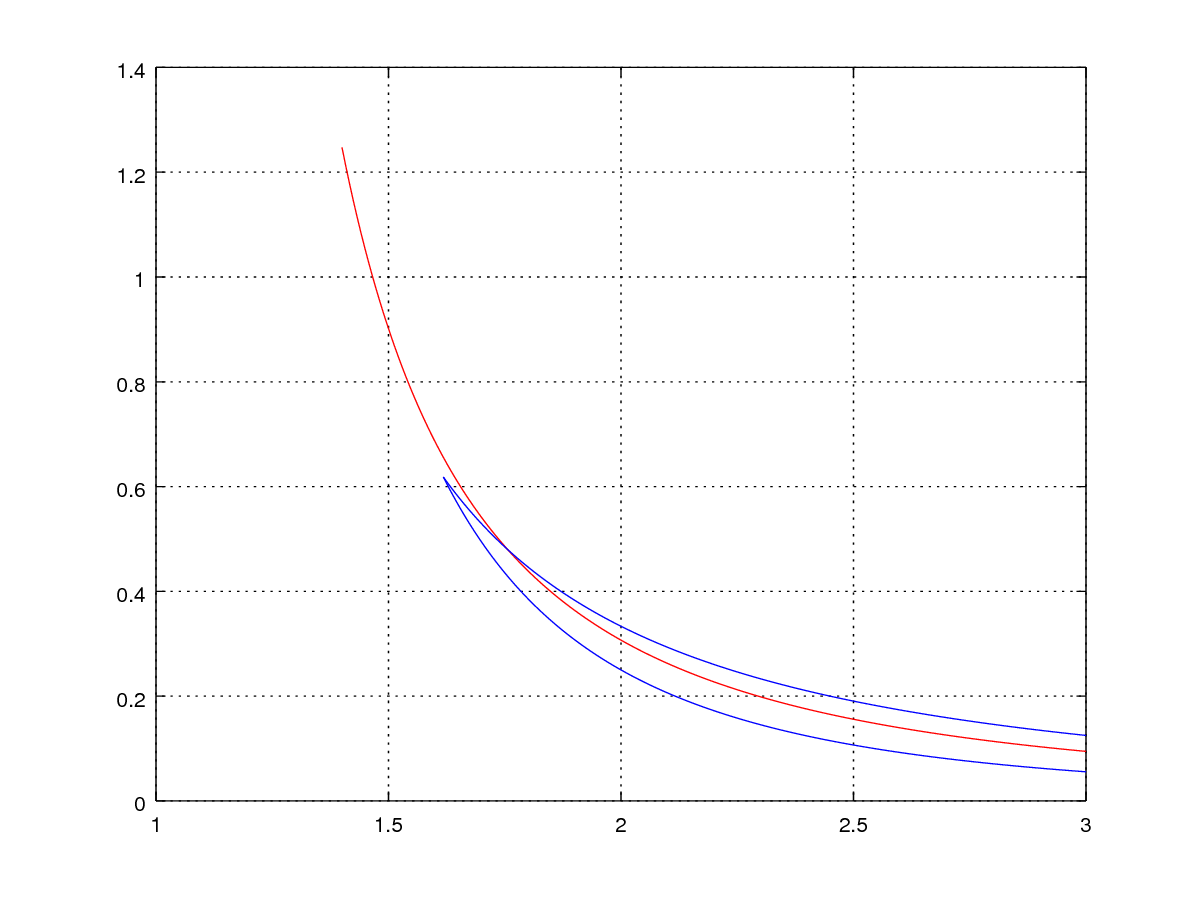} 

\begin{Thm}\label{thm_mean}
\begin{itemize}
 \item[(i)]
 \[
  \lim_{N\to\infty} m_N(q) = \frac{1}{q-1} + \ln\left(\frac{q-1}{q}\right)
 \]
 \item[(ii)] For all $n\in\N$ with $n>1$ it holds that
 \[
  \begin{cases}
   \|n\|_q>\frac{1}{q^2-1} & \textrm{for }n\textrm{ odd and}\\
   \|n\|_q\leq\frac{1}{q^2(q-1)} & \textrm{for }n\textrm{ even.}
  \end{cases}
 \]
 But only for $q>\Phi=\frac{1+\sqrt{5}}{2}$ we have $\frac{1}{q^2(q-1)}<\frac{1}{q^2-1}$ and only for $q>\lambda$
 \[
  \frac{1}{q^2(q-1)}<\lim_{N\to\infty} m_N(q)<\frac{1}{q^2-1}
 \]
 holds. Here $\lambda\approx1.75$ denotes the unique solution of the equation \[\frac{\lambda}{\lambda-1}=(\lambda+1)\ln\left(\frac{\lambda}{\lambda-1}\right).\]
 
\end{itemize}
 
\end{Thm}

\begin{proof}
(i) Due to the computations using the formula for the geometric series in Example \ref{ex_norm}, it is possible to write the $q$-norm as
\[
 \|n\|_q = \sum_{k\in\N,k\nmid n}\frac{1}{q^k} = \frac{1}{q(q-1)}-\sum_{k\geq2,k\mid n}\frac{1}{q^k}
\]
for all $n\in\Z\setminus\{0\}$. Hence for the mean value $m_N(q)$ we get
\[
 m_N(q) = \frac{\displaystyle\sum_{n=1}^N\|n\|_q}{N} = \frac{\displaystyle\frac{N}{q(q-1)}-\displaystyle\sum_{n=2}^N\sum_{k\geq2,k\mid n}\frac{1}{q^k}}{N}.
\]
For all numbers in $\{2,\dots,N\}$ every number $i\in\{2,\dots,\lfloor\frac{N}{2}\rfloor\}$, appears as a divisor exactly $\lfloor\frac{N}{i}\rfloor$ times and the numbers $\lfloor\frac{N}{2}\rfloor+1$ to $N$ appear as divisors exactly once. Thus, assuming $N\geq2$, we get
\[
 m_N(q) = \frac{1}{q(q-1)} - \frac{1}{N}\left(\sum_{k=2}^{\lfloor\frac{N}{2}\rfloor}\left\lfloor\frac{N}{k}\right\rfloor\frac{1}{q^k}+\sum_{j=\lfloor\frac{N}{2}\rfloor+1}^N\frac{1}{q^j}\right).
\]
For all $k\in\{2,\dots,\lfloor\frac{N}{2}\rfloor\}$ there is a number $c_k\in[0,1)$ such that
\[
 \lfloor\frac{N}{k}\rfloor = \frac{N}{k}-c_k
\]
and thus
\[
 \frac{\lfloor\frac{N}{k}\rfloor}{N}=\frac{1}{k}-\frac{c_k}{N}.
\]
Hence we get 
\[
 m_N(q) = \frac{1}{q(q-1)} - \sum_{k=2}^{\lfloor\frac{N}{2}\rfloor}\frac{1}{kq^k} - \frac{1}{N}\sum_{k=2}^{\lfloor\frac{N}{2}\rfloor}\frac{c_k}{q^k}
 - \frac{1}{N}\sum_{j=\lfloor\frac{N}{2}\rfloor+1}^N\frac{1}{q^j}.
\]
Now we have due to the formula for the geometric series
\[
 0\leq\frac{1}{N}\sum_{k=2}^{\lfloor\frac{N}{2}\rfloor}\frac{c_k}{q^k}\leq\frac{1}{N}\sum_{k=2}^{\lfloor\frac{N}{2}\rfloor}\frac{1}{q^k}
 =\frac{1}{N}\left(\frac{\frac{1}{q}-\left(\frac{1}{q}\right)^{\lfloor\frac{N}{2}\rfloor+1}}{1-\frac{1}{q}}-\frac{q-1}{q}\right)\stackrel{N\to\infty}{\longrightarrow}0
\]
and thus
\[
 \lim_{N\to\infty}\frac{1}{N}\sum_{k=2}^{\lfloor\frac{N}{2}\rfloor}\frac{c_k}{q^k} = 0.
\]
Furthermore
\[
 \frac{1}{N}\sum_{j=\lfloor\frac{N}{2}\rfloor+1}^N\frac{1}{q^j} \leq \frac{1}{N}\sum_{j=0}^N\frac{1}{q^j}= \frac{1}{N}\left(\frac{\frac{1}{q}-\left(\frac{1}{q}\right)^{N+1}}{1-\frac{1}{q}}\right)\stackrel{N\to\infty}{\longrightarrow}0.
\]
Hence we get
\[
 \lim_{N\to\infty}m_N(q) = \frac{1}{q(q-1)}-\lim_{N\to\infty}\sum_{k=2}^{\lfloor\frac{N}{2}\rfloor}\frac{1}{kq^k}=\frac{1}{q(q-1)} - \lim_{N\to\infty}\left(\sum_{k=1}^{N}\frac{1}{kq^k}-\frac{1}{q}\right)
\]
\[
 =\frac{1}{q(q-1)} + \frac{1}{q} - \ln\left(\frac{q}{q-1}\right)=\frac{1}{q-1}+\ln\left(\frac{q-1}{q}\right)
\]
using the well known series expansion
\[
 \ln\left(\frac{x}{x-1}\right) = \sum_{n=1}^\infty\frac{1}{nx^n}
\]
of the logarithm for $x>1$.

(ii) For odd $n>1$, we have
\[
 \sum_{k\mid n,k\geq3}\frac{1}{q^k}<\sum_{k=1}^\infty\frac{1}{q^{2k+1}}=\frac{1}{q}\sum_{k=1}^\infty\left(\frac{1}{q^2}\right)^k
 =\frac{1}{q}\left(\frac{1}{1-\frac{1}{q^2}}-1\right) = \frac{1}{q(q^2-1)},
\]
which implies
\[
 \|n\|_q = \frac{1}{q(q-1)}-\sum_{k\mid n,k\geq 3}\frac{1}{q^k}>\frac{1}{q(q-1)}-\frac{1}{q(q^2-1)}=\frac{1}{q^2-1}.
\]
For even $n\in\N$ we get 
\[
 \|n\|_q = \frac{1}{q(q-1)} - \sum_{k\mid n,k\geq 2}\frac{1}{q^k} \leq \frac{1}{q(q-1)} - \frac{1}{q^2} = \frac{1}{q^2(q-1)}.
\]
Solving the inequality
\[
 \frac{1}{q^2-1}>\frac{1}{q^2(q-1)}
\]
for $q>1$ yields
\[
 q>\Phi=\frac{1+\sqrt{5}}{2}.
\]
The equation
\[
 \frac{1}{\lambda^2-1} = \frac{1}{\lambda-1}+\ln\left(\frac{\lambda-1}{\lambda}\right)
\]
is equivalent to 
\[
 \frac{\lambda}{\lambda-1}=(\lambda+1)\ln\left(\frac{\lambda}{\lambda-1}\right).
\]
Its unique solution for $\lambda>1$ can be computed numerically and is approximately $\lambda\approx 1.75$. For $q>\lambda$, we get 
\[
 \frac{1}{q^2-1} > \frac{1}{q-1}+\ln\left(\frac{q-1}{q}\right)
\]
and thus 
\[
 \lim_{N\to\infty}m_N(q) < \frac{1}{q^2-1}
\]
for $q>\lambda$. The remaining inequality, we prove for $q>\frac{5}{3}$, which suffices as $\lambda>\frac{5}{3}$. We get 
\[
 \frac{1}{3}\frac{1}{q-1}<\frac{1}{2}
\]
from $q>\frac{5}{3}$. This implies
\[
 \sum_{n=1}^\infty\frac{1}{(n+2)q^n}<\frac{1}{3}\sum_{n=1}^\infty\frac{1}{q^n}=\frac{1}{3}\frac{1}{q-1}<\frac{1}{2},
\]
which yields
\[
 \frac{1}{2q^2}>\sum_{n=1}^\infty\frac{1}{(n+2)q^{n+2}}=\sum_{n=3}^\infty\frac{1}{nq^n}=\ln\left(\frac{q}{q-1}\right)-\frac{1}{q}-\frac{1}{2q^2},
\]
by the same series expansion of the logarithms as used above. From this inequality we can deduce
\[
 \frac{q+1}{q^2}>\ln\left(\frac{q}{q-1}\right),
\]
which is equivalent to
\[
 \frac{1}{q^2(q-1)}<\frac{1}{q-1}+\ln\left(\frac{q-1}{q}\right).
\]
\end{proof}

\section{Convergence of Sequences with respect to the $q$-norms}\label{sec_three}

As Example \ref{ex_norm} above shows, there are sequences of positive integers for which the sequence of $q$-norms converges to the maximum $\frac{1}{q(q-1)}$ (e.g., the sequence of prime numbers) and such sequences for which the sequence of $q$-norms converges to the minimum $0$ (e.g., the sequence of factorial numbers). The following example provides two additional standard sequences.

\begin{example}\label{ex_more_sequences}
Let
\[
 p_n\# := \prod_{i=1}^np_i
\]
be the sequence of \textit{primorials}, i.e., the product of the first $n$ prime numbers and 
\[
 \mathrm{lcm}(n) = \prod_{p\ \mathrm{prime}, p\leq n} p^{\lfloor\log_p(n)\rfloor}
\]
the sequence of \textit{least common multiples} of the integers $\{1,2,\dots,n\}$.

Then $\|p_n\#\|_q$ converges to some real number in $\left(0,\frac{1}{q(q-1)}\right)$. As all numbers $p_n\#$ are square-free, there is some prime number $\tilde{p}$ such that $\tilde{p}^2\nmid p_n\#$ for all $n\in\N$. Hence, $\|p_n\#\|_q$ does not converge to zero by Proposition \ref{prop_convergence} below. But as $p_{n+1}\#=p_n\#\cdot p_{n+1}$, the sequence is strictly monotonically decreasing:
\[
 \|p_{n+1}\#\|_q-\|p_n\#\|_q = \sum_{k\geq2,k\mid p_n\#}\frac{1}{q^k}-\sum_{k\geq2,k\mid p_n\#\cdot p_{n+1}}\frac{1}{q^k}<-\frac{1}{q^{p_{n+1}}}<0.
\]
Hence it is convergent to some number in $\left(0,\frac{1}{q(q-1)}\right)$. A numerical computation of this limit for $q=2$ and $n=11$ shows
\[
 \lim_{n\to\infty}\|p_n\#\|_2\approx0.06862364.
\]

And
\[
 \|\mathrm{lcm}(n)\|_q=\sum_{k\nmid \mathrm{lcm}(n), k\in\N}\frac{1}{q^k} < \sum_{k=n+1}^\infty \frac{1}{q^k} = \frac{1}{q-1}\frac{1}{q^n}\stackrel{n\to\infty}{\longrightarrow} 0
\]
as surely $k\mid \mathrm{lcm}(n)$ for all $k\in\{1,\dots,n\}$ similarly to the sequence of factorials in Example \ref{ex_norm} above.
\end{example}

Making use of unique prime factorizations of positive integers (c.f.\ \cite[Thm.\ 2]{Hardy2008}), we give the following

\begin{Def}
A sequence $(a_n)_{n\in\N}$ of integers is called \emph{prime factor increasing} if for every prime number $p\in\N$, there is some index $N(p)\in\N$ such that $p\nmid a_n$ for all $n\geq N(p)$. 
\end{Def}

Informally spoken, a sequence of positive integers is prime factor increasing if the prime factors of its terms eventually become arbitrarily large. A simple example for such a sequence is the sequence
\[
 (p_n)_{n\in\N}=(2,3,5,7,\dots)
\]
of prime numbers itself. But one can also construct more intricated sequences which are prime factor increasing, e.g., a sequence $(a_n)_{n\in\N}$ with
\[
 a_n := \prod_{k=n}^{2n} p_k
\]
for all $n\in\N$, where every term is the product of n increasing prime numbers. In this case also the number of prime factors of the terms gets arbitrarily large with $n\in\N$. But note that also any sequence that becomes eventually stationary equal to $1$ from some index on to infinity is, of course, prime factor increasing. 

Using the Definition above we can characterize the convergence of certain sequences in the $q$-norms. Informally spoken, the following Proposition states, that a sequence of positive integers converges to zero in the $q$-norm if its terms are eventually divisible by all positive integers and it converges to the maximum $\frac{1}{q(q-1)}$ if its terms are eventually divisible by no integer. 

\begin{Prop}\label{prop_convergence}
Let $(a_n)_{n\in\N}$ be a sequence of positive integers. Then we have:
\begin{itemize}
 \item[(i)] $\|a_n\|_q\stackrel{n\to\infty}{\longrightarrow}\frac{1}{q(q-1)}$ if and only if $(a_n)_{n\in\N}$ is prime factor increasing.
 \item[(ii)] $\|a_n\|_q\stackrel{n\to\infty}{\longrightarrow}0$ if and only if for every prime number $p\in\N$ and every $K\in\N$ there is some $N(p,K)\in\N$ such that $p^K\mid a_n$ for all $n\geq N(p,K)$. 
 \item[(iii)] $\|a_n\|_q\stackrel{n\to\infty}{\longrightarrow}0$ if and only if for every $k\in\N$ there is some $\nu(k)\in\N$ such that $k\mid a_n$ for all $n\geq\nu(k)$.
\end{itemize}
\end{Prop}

\begin{proof}
(i) Firstly, assume $(a_n)_{n\in\mathbb{N}}$ is prime factor increasing. We have to show that $\|a_n\|_q\stackrel{n\to\infty}{\longrightarrow}\frac{1}{q(q-1)}$, i.e., 
\[
 \forall \varepsilon > 0\ \exists N(\varepsilon) \in \mathbb{N}\ \forall n \geq N(\varepsilon)\colon \left| \frac{1}{q(q-1)} - ||a_n||_q\right| < \varepsilon.
\]
Let $\varepsilon > 0$ and denote by $p_{\varepsilon}$ the smallest prime number which satisfies 
\[
 \log_q\left(\frac{q}{(q-1)\varepsilon}\right) \leq p_{\varepsilon}.
\]
As $(a_n)_{n\in\N}$ is prime factor increasing, there is an index $1<N(\varepsilon)\in\N$ such that $p\nmid a_n$ for all prime numbers $p\leq p_\varepsilon$ and all $n\geq N(\varepsilon)$. Now let \[M = \{ m \in \mathbb{N} \mid |a_m| \neq 1\},\] hence for all $n\geq N(\varepsilon)$, we have
\[
 \left| \frac{1}{q(q-1)} - ||a_n||_q\right| = 0
\]
if $n\notin M$ (c.f. Example \ref{ex_metric}) and
\[
 \left| \frac{1}{q(q-1)} - ||a_n||_q\right| = \sum_{k\mid a_n, k\geq2}\frac{1}{q^{k}} < \sum_{k=p_\varepsilon}^\infty\frac{1}{q^{k}} = \frac{q}{q-1}\frac{1}{q^{p_\varepsilon}} \leq \varepsilon
\]
if $n\in M$.

Secondly, assume that $\|a_n\|_q\stackrel{n\to\infty}{\longrightarrow}\frac{1}{q(q-1)}$ and $(a_n)_{n\in\N}$ is not prime factor increasing. In this case, we have
 \[
  \frac{1}{q(q-1)} - \|a_n\|_q = \sum_{k\mid a_n, k\geq2}\frac{1}{q^{k}}\stackrel{n\to\infty}{\longrightarrow}0
 \]
and there exists a prime number $p\in\N$ such that for all $N\in\N$ there is some $n\geq N$ with $p\mid a_n$. Hence, there is some sub-sequence $(a_{n_j})_{j\in\N}$ of $(a_n)_{n\in\N}$, such that $p\mid a_{n_j}$ for all $j\in\N$ and we have
 \[
  \sum_{k\mid a_{n_j}, k\geq 2}\frac{1}{q^{k}}\geq \frac{1}{q^p}
 \]
for all $j\in\N$ contradicting the convergence $\|a_n\|_q\stackrel{n\to\infty}{\longrightarrow}\frac{1}{q(q-1)}$.

Now we show items (ii) and (iii) in the following way: Firstly, we show that the characterization of the convergence given in (ii) implies the one in (iii). Let $k\in\N$ have the unique prime factorization 
\[
 k=\prod_{i=1}^{\alpha(k)}p_i(k)^{e_i(k)},
\]
with the number of prime factors $\alpha(k)$ of $k$, the prime factors $p_i(k)$ of $k$ and their multiplicity $e_i(k)$ for $i\in\{1,\dots,\alpha(k)\}$. Then set
\[
 \nu(k) := \max\{N(p_i(k),e_i(k))\,|\,i\in\{1,\dots,\alpha(k)\}\},
\]
which yields
\[
 p_i(k)^{e_i(k)}\mid a_n
\]
for all $i\in\{1,\dots,\alpha(k)\}$ and all $n\geq\nu(k)$ and thus
\[
 k\mid a_n
\]
for all $n\geq\nu(k)$. Secondly, we show that the condition in item (iii) indeed implies the convergence of the sequence. Let
\[
 \iota(k) := \max\{\nu(i)\,|\,i\in\{1,\dots,k\}\}.
\]
Then for all $i\in\{1,\dots,k\}$ and all $n\geq\iota(k)$ we have $i\mid a_n$. Hence
\[
 \|a_n\|_q=\sum_{i\in\N,i\nmid a_n}\frac{1}{q^i}\leq\sum_{i=k+1}^\infty\frac{1}{q^i} = \frac{1}{q-1}\frac{1}{q^k}.
\]
As $k\in\N$ is arbitrary, we have
\[
 \forall\,k\in\N\ \exists\,\iota(k)\in\N\ \forall\,n\geq\iota(k)\colon \|a_n\|_q\leq\frac{1}{q-1}\frac{1}{q^k}.
\]
We have to show that $\|a_n\|_q\stackrel{n\to\infty}{\longrightarrow}0$, i.e.,
\[
 \forall\,\varepsilon>0\ \exists\,N(\varepsilon)\in\N\ \forall\,n\geq N(\varepsilon)\colon \|a_n\|_q\leq\varepsilon.
\]
Thus let $\varepsilon>0$. We choose
\[
 k(\varepsilon) = \max\{1,\lceil-\log_q(\varepsilon(q-1))\rceil\}
\]
as well as 
\[
 N(\varepsilon)=\iota(k(\varepsilon)).
\]
Now it is straightforward to see, that in this case
\[
 \|a_n\|_q\leq\frac{1}{q-1}\frac{1}{q^{k(\varepsilon)}}\leq\varepsilon
\]
for all $n\geq N(\varepsilon)$.

Thirdly, we show that the negation of the condition given in item (ii) implies that the sequence does not converge to $0$. This completes the proof. Assume there is some prime $p\in\N$ and some $K\in\N$ such that for all $N\in\N$ there is some $n\geq N$ such that $p^K\nmid a_n$. This implies that there is a subsequence $(a_{n_j})_{j\in\N}$ of $(a_n)_{n\in\N}$ such that $p^K\nmid a_{n_j}$ for all $j\in\N$. By use of Lemma \ref{thelemma} we have in this case
\[
 \|a_{n_j}\|_q\geq\frac{1}{q^{p^K}}
\]
for all $j\in\N$, hence $(\|a_n\|_q)_{n\in\N}$ does not converge to $0$. 
 
\end{proof}

The following Proposition gives a connection between the convergence of a sequence of positive integers with respect to the F\"urstenberg topology $\mathcal{T}$ and the convergence of the sequence of its $q$-norms.

\begin{Prop}\label{prop_convergence2}
Let $(a_n)_{n\in\N}$ be a sequence of positive integers. Then we have:
\begin{itemize}
 \item[(i)] $\|a_n\|_q\stackrel{n\to\infty}{\longrightarrow}0$ if and only if $(a_n)_{n\in\N}$ converges to $0$ with respect to the F\"urstenberg topology $\mathcal{T}$.
 \item[(ii)] If $\|a_n\|_q\stackrel{n\to\infty}{\longrightarrow}0$, then
 \begin{itemize}
  \item[$\alpha$)] $\|a_n+1\|_q\stackrel{n\to\infty}{\longrightarrow}\frac{1}{q(q-1)}$ and
  \item[$\beta$)] $(a_n+1)_{n\in\N}$ converges to $1$ with respect to the F\"urstenberg topology $\mathcal{T}$.
 \end{itemize}
 \item[(iii)] If $(a_n)_{n\in\N}$ converges to $1$ with respect to the F\"urstenberg topology $\mathcal{T}$, then $\|a_n\|_q\stackrel{n\to\infty}{\longrightarrow}\frac{1}{q(q-1)}$.
\end{itemize}

\end{Prop}

Note that item (i) from this proposition together with item (iii) from Prop.~\ref{prop_convergence} implies that a sequence $(a_n)_{n\in\N}$ converges to $0$ with repsect to the F\"urstenberg topology if an only if for all $k\in\N$ there is some $\nu\in\N$ such that $k\,|\,a_n$ for all $n\geq\nu$, which is also given as a characterization of this type of convergence in \cite[Thm.~4]{Lovas2015}.  

\begin{proof}
(i) Convergence with respect to $\mathcal{T}$ is equivalent to convergence with respect to the metrics $d_q$. Hence, a sequence of positive integers $(a_n)_{n\in\N}$ converges to $0$ with respect to $\mathcal{T}$ if and only if $d_q(a_n,0)\stackrel{n\to\infty}{\longrightarrow}0$. As $d_q(a_n,0)=\|a_n-0\|_q=\|a_n\|_q$ the claim follows. 

(ii) $\alpha$) As $\|a_n\|_q\stackrel{n\to\infty}{\longrightarrow}0$, we have by Proposition \ref{prop_convergence}:
\[
 \forall\,k\geq 2\ \exists\,\nu(k)\in\N\ \forall\,n\geq\nu(k):\ k\mid a_n 
\]
\[
 \Rightarrow\forall\,k\geq 2\ \exists\,\nu(k)\in\N\ \forall\,n\geq\nu(k):\ k\nmid a_n+1 
\]
in particular, for all prime numbers $p\in\N$ there is some $\nu(p)\in\N$ such that $p\nmid a_n+1$ for all $n\geq\nu(p)$, hence $(a_n+1)_{n\in\N}$ is prime factor increasing and thus $\|a_n+1\|_q\stackrel{n\to\infty}{\longrightarrow}\frac{1}{q(q-1)}$.

$\beta$) We have
\[
 d_q(a_n+1,1)=\|a_n+1-1\|_q=\|a_n\|_q\stackrel{n\to\infty}{\longrightarrow}0
\]
and thus $(a_n+1)_{n\in\N}$ converges to $1$ with respect to $\mathcal{T}$.

(iii) The sequence $(a_n)_{n\in\N}$ converging to $1$ with respect $\mathcal{T}$ implies 
\[
 d_q(a_n,1)=\|a_n-1\|_q\stackrel{n\to\infty}{\longrightarrow}0.
\]
Hence, by item (ii), we have
\[
 \|a_n\|_q=\|(a_n-1)+1\|_q\stackrel{n\to\infty}{\longrightarrow}\frac{1}{q(q-1)}.
\]

\end{proof}

\begin{remark}
Note, however, that not all integer sequences $(a_n)_{n\in\N}$, which obey
\[
 \|a_n\|_q\stackrel{n\to\infty}{\longrightarrow}\frac{1}{q(q-1)}
\]
do converge to $1$ with respect to the F\"urstenberg topology $\mathcal{T}$. For example the sequence of prime numbers $(p_n)_{n\in\N}$ obey
$\|p_n\|_q\stackrel{n\to\infty}{\longrightarrow}\frac{1}{q(q-1)}$. But by Dirichlet's theorem on arithmetic progressions (\cite{Dirichlet1837}, \cite[Thm. 15]{Hardy2008}) there are infinitely many prime numbers congruent to $3$ modulo $4$. For every prime number $p\equiv3\mod 4$, we get $p-1\equiv 2\mod 4$ and thus $4\nmid p-1$, which implies $\|p-1\|_q\geq\frac{1}{q^4}$ by Lemma \ref{thelemma}. Therefore, let $(p_{n_k})_{k\in\N}$ be the subsequence of the prime numbers, which are all congruent to $3$ modulo $4$. Then $\|p_{n_k}-1\|_q\nrightarrow 0$ as $k\to\infty$, which yields $d_q(p_n,1)\nrightarrow 0$ as $n\to\infty$, and thus the prime numbers $(p_n)_{n\in\N}$ cannot converge to $1$ with respect to $\mathcal{T}$.
\end{remark}

\begin{Prop}\label{prop_primecompletion}
Let $(p_n)_{n\in\N}$ be the sequence of prime numbers. There is a subsequence $(p_{n_k})_{k\in\N}$ of prime numbers which converges to $1$ with respect to the F\"urstenberg topology $\mathcal{T}$.  
\end{Prop}

\begin{proof}
We show that $1$ is a cluster point of $(p_n)_{n\in\N}$. As $(\Z,\mathcal{T})$ is clearly first countable, every sequence possessing a cluster point has a subsequence converging to this cluster point. Let $\{1+b\Z\,|\,b\in\N, b>1\}$ be a neighborhood basis of $1$ in $(\Z,\mathcal{T})$. Then by Dirichlet's theorem on arithmetic progressions all sets $1+b\Z$ contain infinitely many prime numbers. Thus $1$ is a cluster point of $(p_n)_{n\in\N}$. 
\end{proof}

\begin{remark}
Note that Prop.~\ref{prop_primecompletion} is actually a special case of \cite[Thm.~4.1]{Broughan2003}. In \cite{Broughan2003} and the follow-up papers \cite{Hernandez2008} and \cite{Broughan2010}, the authors also investigated the behavior of further interesting sequences in the F\"urstenberg topology like the Mersenne primes and the Fibonacci sequence.
\end{remark}

\section{Summations of $q$-Norms}\label{sec_sum}
In this section we will use the \emph{number-of-divisors function} $d\colon\Z\setminus\{0\}\to\N$ defined by
\[
 d(n)=\sum_{k\in\N,k\mid n}1,
\]
the \emph{Riemann zeta-function} $\zeta\colon(1,\infty)\to\R$ defined by
\[
 \zeta(s)=\sum_{n=1}^\infty\frac{1}{n^s}
\]
and the \emph{M\"obius function} $\mu\colon\N\to\Z$ defined by
\[
 \mu(n) = \begin{cases}1,&n=1 \\ 0,& \exists\textrm{ prime number }p\colon p^2\mid n \\ (-1)^k,& n=\prod_{i=1}^kp_i\textrm{ with prime numbers }p_i\\ & \textrm{such that }p_i\neq p_j\textrm{ if }i\neq j,\,\forall i,j\in\{1,\dots,k\}.\end{cases}
\]
The M\"obius function has the following property (cf.\ \cite[Thm. 263]{Hardy2008}):
\[
 \sum_{k\mid n}\mu(k) = \begin{cases}1,&n=1 \\ 0,& n\in\N, n>1.\end{cases}
\]

\begin{Def}
Let the \emph{sum-$q$-norm} $\xi\colon\Z\to\R$ be defined by
\[
 \xi(n)=\sum_{q=2}^\infty\|n\|_q
\]
and let $d_{\xi}\colon\Z\times\Z\to\R$ be the mapping defined by
\[
 d_{\xi}(n,m)=\xi(n-m).
\]

\end{Def}

The mapping $\xi$ is well-defined because of $0\leq\|n\|_q\leq\frac{1}{q(q-1)}$, which implies
\[
 0\leq\sum_{q=2}^\infty\|n\|_q\leq\sum_{q=2}^\infty\frac{1}{q(q-1)}=1
\]
as well as the fact that the sequence $\left(\xi_N(n)\right)_{N\in\N}=\left(\sum_{q=2}^N\|n\|_q\right)_{N\in\N}$ is monotonically increasing for every $n\in\Z$ and thus the series converges for every $n\in\Z$.

\begin{Thm}
The mapping $d_\xi$ is a metric on $\Z$, which induces the F\"urstenberg topology $\mathcal{T}$ on $\Z$. 
\end{Thm}

\begin{proof}
As all $q$-norms have the properties $\|n\|_q\geq 0,\ \forall n\in\Z$ and $\|n\|_q=0\ \Leftrightarrow\ n=0$, as well as $\|n+m\|_q\leq\|n\|_q+\|m\|_q, \forall n,m\in\Z$, we can deduce the same for $\xi$:
\[
 \xi(n)=\sum_{q=2}^\infty\|n\|_q\geq 0\quad\textrm{and}\quad\xi(n)=\sum_{q=2}^\infty\|n\|_q=0\ \Leftrightarrow\ n=0
\]
for all $n\in\Z$ and
\[
 \xi(n+m)=\sum_{q=2}^\infty\|n+m\|_q\leq \sum_{q=2}^\infty\left(\|n\|_q+\|m\|_q\right)=\sum_{q=2}^\infty\|n\|_q + \sum_{q=2}^\infty\|m\|_q = \xi(n)+\xi(m).
\]
Similarly to Proposition \ref{prop_metric}, these conditions imply that $d_\xi$ is indeed a metric on $\Z$.

As all $q$-norms induce the F\"urstenberg topology by Theorem \ref{thetheorem}, it is enough to show that the metric $d_\xi$ is equivalent to one of the metrics $d_q$. For this we choose $d_2$ and show that there are constants $A,B>0$ such that
\[
 A\cdot d_2(n,m)\leq d_\xi(n,m)\leq B\cdot d_2(n,m)
\]
for all $n,m\in\Z$. It is pretty easy to see that we can choose $A=1$ to fulfill the left one of these inequalities, as surely
\[
 d_2(n,m)=\|n-m\|_2\leq\sum_{q=2}^\infty\|n-m\|_q=d_\xi(n,m)
\]
holds for all $n,m\in\Z$. To find the constant $B$ fulfilling the right one of the inequalities we first show that for all $n\in\Z$ and all $q\geq 5$
\[
 q^{\frac{11}{10}}\|n\|_q\leq\|n\|_2
\]
holds. Indeed it can easily been calculated that
\[
 \frac{11}{10}\frac{\ln(5)}{\ln(5)-\ln(2)}\leq 2
\]
and thus, as $\frac{\ln(q)}{\ln(q)-\ln(2)}$ is monotonically decreasing in $q$, we have
\[
 \frac{11}{10}\frac{\ln(q)}{\ln(q)-\ln(2)}\leq k
\]
for all $k\geq 2$ and all $q\geq 5$. This implies
\[
 \left(\frac{q}{2}\right)^k\geq q^{\frac{11}{10}}
\]
and thus
\[
 \frac{1}{2^k}-q^{\frac{11}{10}}\frac{1}{q^k}\geq 0
\]
for all $k\geq 2$ and all $q\geq 5$. Hence, we get for all any $n\in\Z$
\[
 \sum_{k\in\N,k\nmid n}\left(\frac{1}{2^k}-q^{\frac{11}{10}}\frac{1}{q^k}\right) = \|n\|_2-q^{\frac{11}{10}}\|n\|_q\geq 0
\]
and the desired inequality follows for all $k\geq 2$ and all $q\geq 5$. Furthermore, certainly it holds for all $n\in\Z$ and all $q\in\N$ with $q\geq2$, that
\[
 \|n\|_q\leq\|n\|_2.
\]
With these results established, we can compute
\[
 d_\xi(n,m)=\sum_{q=2}^\infty\|n-m\|_q=\sum_{q=2}^4\|n-m\|_q+ \sum_{q=5}^\infty\|n-m\|_q
\]
\[
 \leq 3\|n-m\|_2 + \sum_{q=5}^\infty\frac{\|n-m\|_2}{q^{\frac{11}{10}}} 
 = \underbrace{\left(2+\zeta\left(\frac{11}{10}\right)-2^{-\frac{11}{10}}-3^{-\frac{11}{10}}-4^{-\frac{11}{10}}\right)}_{=:B}d_2(n,m).
\]
Hence, also $d_\xi$ induces the F\"urstenberg topology on $\Z$.
\end{proof} 

\begin{Prop}\label{prop_sigmazeta}
For $n\in\Z\setminus\{0\}$, we have
\[
 d(n)=\xi(n)+\sum_{k\geq2,k\mid n}\zeta(k)
\]
and for $n\in\N$, $n>1$ also
\[
 \zeta(n) = \sum_{k\in\N,k\mid n}\mu(k)\left(d\left(\frac{n}{k}\right)-\xi\left(\frac{n}{k}\right)\right)
\]
holds.
\end{Prop}

\begin{proof}
For all $n\in\Z\setminus\{0\}$, we compute
\[
 \xi(n)=\sum_{q=2}^\infty\|n\|_q=\sum_{q=2}^\infty\left(\frac{1}{q(q-1)}-\sum_{k\geq2,k\mid n}\frac{1}{q^k}\right)=1-\sum_{k\geq2,k\mid n}\sum_{q=2}^\infty\frac{1}{q^k}
\]
\[
 =1-\sum_{k\geq2,k\mid n}\left(\zeta(k)-1\right)=1+\sum_{k\geq2,k\mid n}1-\sum_{k\geq2,k\mid n}\zeta(k)=d(n)-\sum_{k\geq2,k\mid n}\zeta(k)
\]
and the desired first formula follows. The second formula is basically a M\"obius inversion (cf.\ \cite[Sec. 16.4]{Hardy2008}) of the first one, but we have to note, that the summation in the first formula runs only over divisors greater than $1$. To establish that the formula holds in this case, we mimic the proof of Thm. 266 in \cite{Hardy2008} for our particular case. Using the established first formula we get for $n\in\N$, $n>1$
\[
 \sum_{k\in\N,k\mid n}\mu(k)\left(d\left(\frac{n}{k}\right)-\xi\left(\frac{n}{k}\right)\right) = \sum_{k\in\N,k\mid n}\mu(k)\sum_{m\geq2,m\mid \frac{n}{k}}\zeta(m)
\]
\[
 =\sum_{k\in\N,m\geq2,km\mid n}\mu(k)\zeta(m) = \sum_{m\geq 2, m\mid n}\zeta(m)\sum_{k\in\N, k\mid\frac{n}{m}}\mu(k).
\]
Now the sum $\sum_{k\in\N, k\mid\frac{n}{m}}\mu(k)$ is $1$ only if $\frac{n}{m}=1$, i.e., if $m=n$ and it is $0$ otherwise. Hence, the desired second formula follows. 
\end{proof}

\begin{corollary}
For any prime number $p$ we have
\[
 \xi(p)+\zeta(p)=2.
\]
\end{corollary}

\begin{proof}
Using Proposition \ref{prop_sigmazeta}, we get
\[
 \xi(p) = d(p)-\sum_{k\geq2,k\mid p}\zeta(k) = 2 - \zeta(p)
\]
and the result follows.
\end{proof}

Making use of Proposition \ref{prop_sigmazeta} it is now possible to forget about the sum-$q$-norm for a moment and establish the astonishing result that the difference of the number-of-divisors function and a certain sum over values of the Riemann zeta function has the properties of a metric on $\Z$, which induces the F\"urstenberg topology.

\begin{corollary}
Let the mapping $f\colon\Z\to\R$ be defined by
\[
 f(n) = \begin{cases}0,&\textrm{if }n=0 \\ d(n) - \sum_{k\geq2,k\mid n}\zeta(k),&\textrm{otherwise}. \end{cases}
\]
This mapping has the following properties:
\begin{itemize}
 \item $0\leq f(n)\leq 1$ for all $n\in\Z$ and $f(n)=0\Leftrightarrow n=0$, as well as
 \item $f(n+m)\leq f(n)+f(m)$, i.e., $f$ is sub-additive.
\end{itemize}
Furthermore, let the mapping $\delta\colon\Z\times\Z\to\R$ be defined by
\[
 \delta(n,m)=f(n-m).
\]
This mapping has the following properties:
\begin{itemize}
 \item $(\Z,\delta)$ is a metric space of diameter $1$, and
 \item $\delta$ induces the F\"urstenberg topology on $\Z$.
\end{itemize}
\end{corollary}

\begin{proof}
The inequality $0 \leq f(n) \leq 1$ and $f(n)=0\Leftrightarrow n=0$, as well as the sub-additivity follow easily from observing that $f(n)=\xi(n)$. Obviously, $\delta=d_\xi$, which implies the remaining claims.
\end{proof}

We will now establish a result for the mean value of the sum-$q$-norm similarly to Theorem \ref{thm_mean} for the mean value of the $q$-norms.

\begin{Thm}\label{thm_mean_sum}
Let $\gamma$ be the Euler-Mascheroni constant. Then we have
\[
 \lim_{N\to\infty}\frac{\sum_{n=1}^N\xi(n)}{N} = \gamma = \lim_{N\to\infty}\frac{\sum_{n=1}^N\left(d(n)-\sum_{k\geq2,k\mid n}\zeta(k)\right)}{N}.
\]
\end{Thm}

\begin{proof}
We have to compute
\[
 \lim_{N\to\infty}\frac{\sum_{n=1}^N\xi(n)}{N}=\lim_{N\to\infty}\frac{1}{N}\sum_{n=1}^N\lim_{Q\to\infty}\sum_{q=2}^Q\|n\|_q=
 \lim_{N\to\infty}\sum_{q=2}^\infty m_N(q),
\]
where 
\[
 m_N(q)=\frac{\sum_{n=1}^N\|n\|_q}{N}
\]
as in Theorem \ref{thm_mean}. As $\|n\|_q\leq\frac{1}{q(q-1)}$ for all $n\in\N$, we observe that $m_N(q)$ is dominated by
\[
 0\leq m_N(q)\leq\frac{1}{q(q-1)},\quad \forall q\in\N\setminus\{1\}
\]
and note that $\frac{1}{q(q-1)}$ is summable over $\N\setminus\{1\}$, i.e., $\sum_{q=2}^\infty\frac{1}{q(q-1)}<\infty$. This makes it possible to apply the theorem of dominated convergence for sums and swap the limit for $N\to\infty$ with the sum from $2$ to $\infty$.

Hence, using the limit
\[
 \lim_{N\to\infty}m_N(q)=\frac{1}{q-1}+\ln\left(\frac{q-1}{q}\right)
\]
from Theorem \ref{thm_mean} as well as some $\alpha(r)\in\mathcal{O}(\frac{1}{r})$, such that $\sum_{q=1}^r\frac{1}{q}=\ln(r)+\gamma+\alpha(r)$, we get
\[
 \lim_{N\to\infty}\frac{\sum_{n=1}^N\xi(n)}{N}=\sum_{q=2}^\infty\lim_{N\to\infty}m_N(q) = \sum_{q=2}^\infty\left(\frac{1}{q-1} + \ln\left(\frac{q-1}{q}\right)\right)
\]
\[
 =\lim_{r\to\infty}\sum_{q=2}^{r+1}\left(\frac{1}{q-1} + \ln(q-1)-\ln(q)\right) = \lim_{r\to\infty}\left(\sum_{q=1}^r\frac{1}{q} - \ln(r+1)\right)
\]
\[
 =\lim_{r\to\infty}\left(\ln\left(\frac{r}{r+1}\right)+\gamma+\alpha(r)\right)=\gamma.
\]
\end{proof}

\begin{Def}\label{def_integralqnorm}
Let the \emph{integral-$q$-norm} $I\colon\Z\to\R$ be defined by
\[
 I(n)=\int_{q=2}^\infty\|n\|_q\,\D q
\]
and let $d_{I}\colon\Z\times\Z\to\R$ be the mapping defined by
\[
 d_{I}(n,m)=I(n-m).
\]

\end{Def}

The mapping $I$ is well-defined because of $0\leq\|n\|_q\leq\frac{1}{q(q-1)}$, which implies that the integral converges for every $n\in\Z$ by the theorem on dominated convergence.

\begin{Prop}\label{prop_integralqnorm}
For the integral-$q$-norm
\[
 I(n)=\ln(2)-\sum_{k\geq2,k\mid n}\frac{2}{k-1}\frac{1}{2^k}=\sum_{k\in\N,k\nmid n}\frac{2}{k-1}\frac{1}{2^k}
\]
holds for all $n\in\Z$.
\end{Prop}

\begin{proof}
Performing the integration from the definition of the integral-$q$-norm yields
\[
 I(n)=\int_{q=2}^\infty\|n\|_q\,\D q = \int_{q=2}^\infty\left(\frac{1}{q(q-1)}-\sum_{k\geq2,k\mid n}\frac{1}{q^k}\right)\,\D q 
\]
\[
 =\int_{q=2}^\infty\frac{\D q}{q(q-1)}-\sum_{k\geq2,k\mid n}\int_{q=2}^\infty\frac{\D q}{q^k} = -\ln\left(\frac{1}{2}\right)-\sum_{k\geq2,k\mid n}\frac{1}{k-1}\frac{1}{2^{k-1}}
\]
\[
 =\ln(2)-\sum_{k\geq2,k\mid n}\frac{2}{k-1}\frac{1}{2^k}.
\]
Observing that
\[
 \sum_{k\geq2,k\mid n}\frac{2}{k-1}\frac{1}{2^k} + \sum_{k\in\N,k\nmid n}\frac{2}{k-1}\frac{1}{2^k} = \sum_{k=2}^\infty\frac{2}{k-1}\frac{1}{2^k}=\sum_{k=1}^\infty\frac{1}{k2^k}=\ln(2)
\]
yields the second equality.
\end{proof}

\begin{Thm}
The mapping $d_I$ is a metric on $\Z$, which induces the F\"urstenberg topology $\mathcal{T}$ on $\Z$. 
\end{Thm}

\begin{proof}
As all $q$-norms have the properties $\|n\|_q\geq 0,\ \forall n\in\Z$ and $\|n\|_q=0\ \Leftrightarrow\ n=0$, as well as $\|n+m\|_q\leq\|n\|_q+\|m\|_q, \forall n,m\in\Z$, we can deduce the same for $I$:
\[
 I(n)=\int_{q=2}^\infty\|n\|_q\,\D q\geq 0\quad\textrm{and}\quad I(n)=\int_{q=2}^\infty\|n\|_q\,\D q=0\ \Leftrightarrow\ n=0
\]
for all $n\in\Z$ and
\[
 I(n+m)=\int_{q=2}^\infty\|n+m\|_q\,\D q\leq \int_{q=2}^\infty\left(\|n\|_q+\|m\|_q\right)\D q
\]
\[
 =\int_{q=2}^\infty\|n\|_q\,\D q + \int_{q=2}^\infty\|m\|_q\,\D q = I(n)+I(m).
\]
Similarly to Proposition \ref{prop_metric}, these conditions imply that $d_I$ is indeed a metric on $\Z$.

It is enough to show that the metric $d_I$ obeys
\[
 d_4(n,m)\leq d_I(n,m)\leq 2\cdot d_2(n,m)
\]
for all $n,m\in\Z$. As all metrics $d_q$ for $q>1$ induce the F\"urstenberg topology on $\Z$ by Theorem \ref{thetheorem}, we have the following: Let $A$ be any subset of the integers. If $A$ is open with respect to $\mathcal{T}$, it is open with respect to $d_2$ and thus it is also open with respect to $2\cdot d_q$. The right hand inequality then implies, that $A$ is also open with respect to $d_I$. If $A$ is open with respect to $d_I$, the left hand inequality implies that it is also open with respect to $d_4$, hence also with respect to $\mathcal{T}$. Hence, we can deduce that any set $A\subset\Z$ is open with respect to $d_I$ if and only if it is open with respect to $\mathcal{T}$. Thus, $\mathcal{T}$ and $d_I$ possess the same system of open sets, i.e., $d_I$ induces $\mathcal{T}$. 

The right hand inequality is pretty easy to establish, as surely by Proposition \ref{prop_integralqnorm} we have
\[
 d_I(n,m) = \sum_{k\in\N,k\nmid n-m}\frac{2}{k-1}\frac{1}{2^k}\leq 2\sum_{k\in\N,k\nmid n-m}\frac{1}{2^k}=2\cdot d_2(n,m)
\]
for all $n,m\in\Z$, due to the fact that 
\[
 \frac{2}{k-1}\leq 2
\]
for all $k\in\N$ with $k>1$.

For the left hand inequality, we observe that
\[
 k-1 \leq 2\cdot 2^k 
\]
for all $k\in\N$, $k>1$. From this we get
\[
 1\leq \frac{2}{k-1}\left(\frac{4}{2}\right)^k
\]
and subsequently
\[
 \frac{1}{4^k}\leq \frac{2}{k-1}\frac{1}{2^k}
\]
for all $k\in\N$, $k>1$. Thus
\[
 d_4(n,m)= \sum_{k\in\N,k\nmid n-m}\frac{1}{4^k} \leq \sum_{k\in\N,k\nmid n-m}\frac{2}{k-1}\frac{1}{2^k} = d_I(n,m).
\]

\end{proof}  

We will now establish a result for the mean value of the integral-$q$-norm similarly to Theorem \ref{thm_mean} for the mean value of the $q$-norms and Theorem \ref{thm_mean_sum} for the sum-$q$-norm.

\begin{Thm}\label{thm_mean_int}
For the integral-$q$-norm
\[
 \lim_{N\to\infty}\frac{\sum_{n=1}^N I(n)}{N} = \ln(4)-1
\]
holds.
\end{Thm}
\begin{proof}
We will make use of 
\[
 m_N(q)=\frac{\sum_{n=1}^N\|n\|_q}{N}
\]
and the limit
\[
 \lim_{N\to\infty}m_N(q)=\frac{1}{q-1}+\ln\left(\frac{q-1}{q}\right)
\]
from Theorem \ref{thm_mean}, which we consider here to be the point-wise limit of the sequence of functions $m_N\colon[2,\infty)\to\R$, $q\mapsto m_N(q)$. As $\|n\|_q\leq\frac{1}{q(q-1)}$ for all $n\in\N$, we observe that the sequence $m_N(q)$ is dominated by 
\[
 0\leq m_N(q)\leq\frac{1}{q(q-1)},\quad \forall q\in[2,\infty)
\]
and note that $\frac{1}{q(q-1)}$ is integrable over $[2,\infty)$, i.e., $\int_{q=2}^\infty\frac{1}{q(q-1)}\D q<\infty$. This makes it possible to apply Lebesgue's dominated convergence theorem and swap the limit for $N\to\infty$ with the integral over $[2,\infty)$.

Hence, we compute
\[
 \lim_{N\to\infty}\frac{\sum_{n=1}^N I(n)}{N}=\lim_{N\to\infty}\frac{\sum_{n=1}^N\int_{q=2}^\infty\|n\|_q\,\D q}{N}=\lim_{N\to\infty}\int_{q=2}^\infty\frac{\sum_{n=1}^N\|n\|_q}{N}\,\D q
\]
\[
 =\int_{q=2}^\infty \lim_{N\to\infty}m_N(q)\,\D q=\int_{q=2}^\infty\left(\frac{1}{q-1}+\ln\left(\frac{q-1}{q}\right)\right)\D q
 =\lim_{c\to\infty}\left.q\ln\left(\frac{q-1}{q}\right)\right|_2^c
\]
\[
 =\lim_{c\to\infty}c\cdot\ln\left(1-\frac{1}{c}\right)-2\cdot\ln\left(\frac{2-1}{2}\right)=-1+\ln(4).
\]

\end{proof}

\section{Generating Functions}\label{sec_gen}
In this section we will make use of the polylogarithm as well as several Dirichlet series and the generating functions of arithmetic functions (c.f. \cite[Ch. XVII]{Hardy2008}), most prominently the Riemann Zeta function, but also the number-of-divisors function and the M\"obius function, all defined at the beginning of Section \ref{sec_sum}. We will start with one further definition and the repetition of some well established results, which we will make use of below.  

\begin{Def}\label{def_polylog}
For all $(s,z)\in (1,\infty)\times (-1,1)$, let the \emph{polylogarithm} be defined by
\[
 \Li(s,z)=\sum_{n=1}^\infty\frac{z^n}{n^s}.
\]
\end{Def}

\begin{remark}\label{rem_polylog}
Note that Definition \ref{def_polylog} entails
\[
 \Li\left(1,\frac{1}{q}\right)=\ln\left(\frac{q}{q-1}\right)
\]
for all $q>1$.
\end{remark}

\begin{Prop}[{\cite[Thms. 287, 289]{Hardy2008}}]\label{prop_dirdiv}
Let $d(n)$ be the number-of-divisors function and $\mu(n)$ the M\"obius function. Then we have
\[
 \zeta^2(s)=\sum_{n=1}^\infty\frac{d(n)}{n^s}
\]
and
\[
 \frac{1}{\zeta(s)}=\sum_{n=1}^\infty\frac{\mu(n)}{n^s}
\]

for all $s>1$.
\end{Prop}

Now we compute the Dirichlet series and the generating function of the $q$-norm $\|\cdot\|_q\colon\N\to\R$ for all $q>1$.

\begin{Thm}
Let
\[
 Q_q(s):=\sum_{n=1}^\infty\frac{\|n\|_q}{n^s}
\]
be the generating function of the $q$-norm for $s>1$ and $q>1$. Then we have
\[
 Q_q(s)=\zeta(s)\left(\frac{1}{q-1}-\Li\left(s,\frac{1}{q}\right)\right).
\]
Furthermore, let
\[
 \gamma_q(n)=\sum_{i\in\N,i\mid n}\mu(i)\left\|\frac{n}{i}\right\|_q
\]
be the M\"obius inversion of the $q$-norm. Then we have
\[
 \sum_{n=1}^\infty\frac{\gamma_q(n)}{n^s} = \frac{1}{q-1}-\Li\left(s,\frac{1}{q}\right)
\]
and in particular for $s=1$
\[
 \sum_{n=1}^\infty\frac{\gamma_q(n)}{n} = \frac{1}{q-1}+\ln\left(\frac{q-1}{q}\right) = \lim_{N\to\infty}\frac{\sum_{n=1}^N\|n\|_q}{N}.
\]
\end{Thm}

\begin{proof}
First we compute
\[
 Q_q(s)=\sum_{n=1}^\infty\frac{\|n\|_q}{n^s}=\sum_{n=1}^\infty\frac{\frac{1}{q(q-1)}-\sum_{k\geq2,k\mid n}\frac{1}{q^k}}{n^s}=\frac{\zeta(s)}{q(q-1)}-\sum_{n=2}^\infty\frac{1}{n^s}\sum_{k\geq2,k\mid n}\frac{1}{q^k}.
\]
As the series $\sum_{n=2}^\infty\frac{1}{n^s}\sum_{k\geq2,k\mid n}\frac{1}{q^k}$ surely is absolutely convergent we can reorder it and get
\[
 Q_q(s)=\frac{\zeta(s)}{q(q-1)}-\sum_{k=2}^\infty\frac{1}{q^k}\sum_{c=1}^\infty\frac{1}{(ck)^s}=\frac{\zeta(s)}{q(q-1)}-\sum_{k=2}^\infty\frac{1}{q^kk^s}\sum_{c=1}^\infty\frac{1}{c^s}
\]
\[
 =\frac{\zeta(s)}{q(q-1)}-\sum_{k=2}^\infty\frac{\zeta(s)}{q^kk^s}=\zeta(s)\left(\frac{1}{q(q-1)}-\sum_{k=1}^\infty\frac{1}{q^kk^s}+\frac{1}{q}\right)
\]
\[
 =\zeta(s)\left(\frac{1}{q-1}-\Li\left(s,\frac{1}{q}\right)\right),
\]
which proves the first claim of the theorem.

Hence, making use of Proposition \ref{prop_dirdiv} and the multiplication of Dirichlet series (cf. \cite[Thm. 284]{Hardy2008}), we get the second claim of the theorem:
\[
 \frac{1}{q-1}-\Li\left(s,\frac{1}{q}\right) = \frac{1}{\zeta(s)}Q_q(s)=\left(\sum_{n=1}^\infty\frac{\mu(n)}{n^s}\right)\cdot\left(\sum_{n=1}^\infty\frac{\|n\|_q}{n^s}\right)
\]
\[
 =\sum_{n=1}^\infty\frac{\sum_{i\in\N,i\mid n}\mu(i)\left\|\frac{n}{i}\right\|_q}{n^s}=\sum_{n=1}^\infty\frac{\gamma_q(n)}{n^s}.
\]
Finally, using Remark \ref{rem_polylog} and Theorem \ref{thm_mean}, we compute for $s=1$
\[
 \sum_{n=1}^\infty\frac{\gamma_q(n)}{n}=\frac{1}{q-1}-\Li\left(1,\frac{1}{q}\right)=\frac{1}{q-1}+\ln\left(\frac{q-1}{q}\right) = \lim_{N\to\infty}\frac{\sum_{n=1}^N\|n\|_q}{N}.
\]

\end{proof}

\begin{corollary}\label{cor_moebius}
The M\"obius inversion $\gamma_q\colon\N\to\R$ for the $q$-norms given by
\[
 \gamma_q(n)=\sum_{i\in\N,i\mid n}\mu(i)\left\|\frac{n}{i}\right\|_q
\]
can be computed explicitely to be
\[
 \gamma_q(n)=\begin{cases}\frac{1}{q-1}-\frac{1}{q},&n=1\\-\frac{1}{q^n},&n>1.\end{cases}
\]
\end{corollary}

\begin{proof}
Writing the $q$-norm as
\[
 \|n\|_q = \frac{1}{q-1}-\sum_{k\in\N,k\mid n}\frac{1}{q^k}
\]
(c.f. Example \ref{ex_norm}) and proceeding similarly to the proof of Proposition \ref{prop_sigmazeta}, we get
\[
 \gamma_q(n)=\sum_{i\in\N,i\mid n}\mu(i)\left\|\frac{n}{i}\right\|_q=\sum_{i\in\N,i\mid n}\mu(i)\left(\frac{1}{q-1}-\sum_{k\in\N,k\mid \frac{n}{i}}\frac{1}{q^k}\right)
\]
\[
 =\frac{1}{q-1}\sum_{i\in\N,i\mid n}\mu(i)-\sum_{i\in\N,i\mid n}\mu(i)\sum_{k\in\N,k\mid \frac{n}{i}}\frac{1}{q^k}=\frac{1}{q-1}\sum_{i\in\N,i\mid n}\mu(i)-\sum_{k\in\N,k\mid n}\frac{1}{q^k}\sum_{i\in\N,i\mid\frac{n}{k}}\mu(i)
\]
\[
 =\begin{cases}\frac{1}{q-1}-\frac{1}{q},&n=1\\-\frac{1}{q^n},&n>1.\end{cases}
\]

\end{proof}

\begin{Prop}\label{prop_sigmavons}
Let 
\[
 \Xi(s) := \sum_{n=1}^\infty\frac{\xi(n)}{n^s}
\]
be the generating function of the sum-$q$-norm for $s>1$. It holds that
\[
 \sum_{n=1}^\infty\left(\frac{1}{n^s}\sum_{k\geq2,k\mid n}\zeta(k)\right) = \zeta(s)\sum_{n=2}^\infty\frac{\zeta(n)}{n^s}
\]
and
\[
 \Xi(s) = \zeta(s)\left(\zeta(s)-\sum_{n=2}^\infty\frac{\zeta(n)}{n^s}\right).
\]
\end{Prop}

\begin{proof}
Reordering the sequence $\sum_{n=1}^\infty(\frac{1}{n^s}\sum_{k\geq2,k\mid n}\zeta(k))$ yields
\[
 \sum_{n=1}^\infty\left(\frac{1}{n^s}\sum_{k\geq2,k\mid n}\zeta(k)\right) = \sum_{n=2}^\infty\left(\frac{1}{n^s}\sum_{k\geq2,k\mid n}\zeta(k)\right)
\]
\[
 =\sum_{n=2}^\infty\left(\zeta(n)\sum_{c=1}^\infty\frac{1}{(cn)^s}\right)=\sum_{n=2}^\infty\left(\frac{\zeta(n)}{n^s}\sum_{c=1}^\infty\frac{1}{c^s}\right) = \zeta(s)\sum_{n=2}^\infty\frac{\zeta(n)}{n^s}.
\]
Making use of Propositions \ref{prop_sigmazeta} and \ref{prop_dirdiv}, we compute
\[
 \Xi(s)=\sum_{n=1}^\infty\frac{d(n)-\sum_{k\geq2,k\mid n}}{n^s} = \sum_{n=1}^\infty\frac{d(n)}{n^s}-\sum_{n=1}^\infty\frac{\sum_{k\geq2,k\mid n}}{n^s} = \zeta(s)^2-\zeta(s)\sum_{n=2}^\infty\frac{\zeta(n)}{n^s}
\]
and the result follows.
\end{proof}

\bibliographystyle{english_custom}
\end{document}